\documentclass{amsart}
\usepackage[utf8]{inputenc}
\usepackage[hidelinks]{hyperref}
\usepackage{amsthm, amsmath, amssymb, enumitem, stmaryrd, verbatim}

\newtheorem{corollary}{Corollary}
\newtheorem{lemma}{Lemma}
\newtheorem{proposition}{Proposition}
\newtheorem{theorem}{Theorem}

\theoremstyle{definition}
\newtheorem{definition}{Definition}

\theoremstyle{remark}
\newtheorem{remark}{Remark}

\DeclareMathOperator{\aut}{Aut}

\begin{document}

\title{On the hom-associative Weyl algebras}
\author{Per B\"ack}
\address[Per B\"ack]{Division of Mathematics and Physics, The School of Education, Culture and Communication, M\"alar\-dalen  University,  Box  883,  SE-721  23  V\"aster\r{a}s, Sweden}
\email[corresponding author]{per.back@mdh.se}

\author{Johan Richter}
\address[Johan Richter]{Department of Mathematics and Natural Sciences, Blekinge Institute of Technology, SE-371 79 Karlskrona, Sweden}
\email{johan.richter@bth.se}


\begin{abstract}
The first (associative) Weyl algebra is formally rigid in the classical sense. In this paper, we show that it can however be formally deformed in a nontrivial way when considered as a so-called hom-associative algebra, and that this deformation preserves properties such as the commuter, while deforming others, such as the center, power associativity, the set of derivations, and some commutation relations. We then show that this deformation induces a formal deformation of the corresponding Lie algebra into what is known as a hom-Lie algebra, when using the commutator as bracket. We also prove that all homomorphisms between any two purely hom-associative Weyl algebras are in fact isomorphisms. In particular, all endomorphisms are automorphisms in this case, hence proving a hom-associative analogue of the Dixmier conjecture to hold true.
\end{abstract}

\maketitle

\section{Introduction}
The study of \emph{hom-associative algebras} has its origins in \emph{hom-Lie algebras}, the latter proposed by Hartwig, Larsson, and Silvestrov~\cite{HLS06} as a generic framework to describe deformations of Lie algebras obeying a generalized Jacobi identity, the latter now twisted by a \emph{hom}omorphism; hence the name. Hom-associative algebras, introduced by Makhlouf and Silvestrov~\cite{MS08}, now play the same role as associative algebras do for Lie algebras; equipping a hom-associative algebra with the commutator as bracket give rise to a hom-Lie algebra. Just as the Jacobi identity in the latter algebras is twisted, the same holds true for the associativity condition in the former. In particular may hom-associative algebras be seen to include associative algebras and general non-associative algebras in the following way: when the map twisting this condition is the identity map, one recovers the associativity condition, and when equal to the zero map, this condition becomes null. We define the purely hom-associative case to be the one in which this map is not a multiple of the identity map. (Note, however, that a purely hom-associative algebra can happen to be associative as well.)

The first (associative) Weyl algebra may be exhibited as an \emph{Ore extension}, or a \emph{non-commutative polynomial ring} as Ore extensions were first named by Ore when he introduced them~\cite{Ore33}. \emph{Non-associative Ore extensions} were later introduced by Nystedt, \"Oinert, and Richter in the unital case~\cite{NOR18}, and then generalized to the non-unital, hom-as\-so\-cia\-tive setting by Silvestrov and the authors~\cite{BRS18}. The authors further developed this theory in~\cite{BR18}, introducing a Hilbert's basis theorem for unital, non-associative and hom-associative Ore extensions. 

The first Weyl algebra is formally rigid in the classical sense of Gerstenhaber who introduced formal deformation theory for associative algebras and rings in the seminal paper~\cite{Ger64}. However, as described above, any associative algebra is a hom-associative algebra with twisting map equal to the identity map, and we show in this paper that as such, it may be formally deformed in a nontrivial way. This results in the \emph{hom-associative Weyl algebras}, first introduced in~\cite{BRS18} along with hom-associative versions of the quantum plane and the universal enveloping algebra of the two-dimensional non-abelian Lie algebra, the latter two which also turned out to be formal deformations of their associative counterparts~\cite{Bac18}. We also show that the formal deformation of the first Weyl algebra into the hom-associative Weyl algebras induce a formal deformation of the corresponding Lie algebra into hom-Lie algebras, when using the commutator as bracket. They can, more formally, be described as \emph{one-parameter formal hom-associative deformations} and \emph{one-parameter formal hom-Lie deformations}, respectively; two notions that have been introduced by Makhlouf and Silvestrov earlier~\cite{MS10}. Furthermore, we see that the former deformation preserves some properties, such as the commuter, while deforming others such as the center, power associativity, the set of derivations, and some commutation relations. The perhaps most interesting fact we are able to prove, however, is that all homomorphisms between any two purely hom-associative Weyl algebras are in fact isomorphisms. In this case, all endomorphisms are therefore automorphisms, hence a hom-associative analogue of the Dixmier conjecture holds true, the latter a still unsolved conjecture that has its origins in a question raised by Dixmier in~\cite{Dix68} (cf. 11. Problèmes). Tsuchimoto~\cite{Tsuchimoto05} and Kanel-Belov and Kontsevich~\cite{KBK07} have moreover been able to prove, independently, that the Dixmier conjecture is stably equivalent to the more famous Jacobian conjecture.  

The paper is organized as follows:

\autoref{sec:prel} focuses on preliminaries on non-associative algebras (\autoref{subsec:general-algebra}), hom-associative algebras and hom-Lie algebras (\autoref{subsec:hom}), non-unital, hom-associative Ore extensions (\autoref{subsec:hom-ore}), the first Weyl algebra (\autoref{subsec:weyl-algebra}), and the hom-associative Weyl algebras (\autoref{subsec:hom-weyl}).

\autoref{sec:morph-assoc-comm} contains results on some of the standard properties of the hom-asso\-cia\-tive Weyl algebras: it is shown that they contain no zero divisors (\autoref{cor:zero-div}), that their commuter is equal to the ground field (\autoref{prop:commuter-of-hom-weyl}), that the first Weyl algebra is the only power associative hom-associative Weyl algebra (\autoref{prop:power-assoc}), and moreover, all their derivations are described (\autoref{cor:derivations}). It is further shown that all homomorphisms between any two purely hom-associative Weyl algebras are isomorphisms of a certain type (\autoref{prop:morphisms}), this implying that all endomorphisms therefore are automorphisms in this case (\autoref{cor:hom-dixmier}).

\autoref{sec:deform} gives a brief review of one-parameter formal hom-associative deformations and one-parameter formal hom-Lie deformations. It is then shown that the hom-associative Weyl algebras are a one-parameter formal hom-associative deformation of the first Weyl algebra (\autoref{prop:weyl-deform}), inducing a one-parameter formal hom-Lie deformation of the Lie algebra when using the commutator as bracket (\autoref{prop:weyl-lie-deform}).  

\newpage

\section{Preliminaries}\label{sec:prel}
Throughout this paper, we denote by $\mathbb{N}$ the nonnegative integers.

\subsection{Non-associative algebras}\label{subsec:general-algebra}
By a \emph{non-as\-so\-cia\-tive algebra} $A$ over an associative, commutative, and unital ring $R$, we mean an $R$-algebra which is not necessarily associative. Furthermore, $A$ is called \emph{unital} if there exists an element $1_A\in A$ such that for all $a\in A$, $a\cdot 1_A=1_A\cdot a=a$, and \emph{non-unital} if there does not necessarily exist such an element. For a non-associative and non-unital algebra $A$, the \emph{commutator} $[\cdot,\cdot]\colon A\times A\to A$ is defined by $[a,b]:=a\cdot b-b\cdot a$ for arbitrary $a,b\in A$, and the \emph{commuter} of $A$, written $C(A)$, as the set $C(A):=\{a\in A\colon [a,b]=0,\ b\in A\}$. The \emph{associator} $(\cdot,\cdot,\cdot)\colon A\times A\times A\to A$ is defined by $(a,b,c)=(a\cdot b)\cdot c-a\cdot (b\cdot c)$ for arbitrary elements $a,b,c\in A$, and the \emph{left}, \emph{middle}, and \emph{right nuclei} of $A$, denoted by $N_l(A), N_m(A),$ and $N_r(A)$, respectively, as the sets $N_l(A):=\{a\in A\colon (a,b,c)=0,\ b,c\in A\}$, $N_m(A):=\{b\in A\colon (a,b,c)=0,\ a,c\in A\}$, and $N_r(A):=\{c\in A\colon (a,b,c)=0,\ a,b\in A\}$. The \emph{nucleus} of $A$, written $N(A)$, is defined as the set $N(A):=N_l(A)\cap N_m(A)\cap N_r(A)$. The \emph{center} of $A$, denoted by $Z(A)$, is the intersection of the commuter and the nucleus, $Z(A):=C(A)\cap N(A)$. If the only two-sided ideals in $A$ are the zero ideal and $A$ itself, $A$ is called \emph{simple}.  A way to measure the non-associativity of $A$ can be done by using the associator: $A$ is called \emph{power associative} if $(a,a,a)=0$, \emph{left alternative} if $(a,a,b)=0$, \emph{right alternative} if $(b,a,a)=0$, \emph{flexible} if $(a,b,a)=0$, and \emph{associative} if $(a,b,c)=0$ for all $a,b,c\in A$. An $R$-linear map $\delta\colon A\to A$ is called a \emph{derivation} if for any $a,b\in A$, $\delta(a\cdot b)=\delta(a)\cdot b+a\cdot\delta(b)$. If $A$ is associative, then all maps of the form $[a,\cdot]\colon A\to A$ for an arbitrary $a\in A$ are derivations called \emph{inner derivations}. If $A$ is not associative, such a map need not be a derivation, however. At last, recall that $A$ \emph{embeds} into a non-associative and non-unital algebra $B$ if there is an injective morphism from $A$ to $B$, the idea being that $A$ now may be seen as a subalgebra of $B$.

\subsection{Hom-associative algebras and hom-Lie algebras}\label{subsec:hom}
This section is devoted to restating some basic definitions and general facts concerning hom-associative algebras and hom-Lie algebras. Though hom-associative algebras as first introduced in \cite{MS08} and hom-Lie algebras in \cite{HLS06} were defined by starting from vector spaces, we take a slightly more general approach here, following the conventions in e.g. \cite{Bac18, BR18, BRS18}, starting from modules; most of the general theory still hold in this latter case, however.

\begin{definition}[Hom-associative algebra]\label{def:hom-assoc-algebra} A \emph{hom-associative algebra} over an associative, commutative, and unital ring $R$, is a triple $(M,\cdot,\alpha)$ consisting of an $R$-module $M$, a binary operation $\cdot\colon M\times M\to M$ linear over $R$ in both arguments, and an $R$-linear map $\alpha\colon M\to M$, satisfying, for all $a,b,c\in M$, $\alpha(a)\cdot(b\cdot c)=(a\cdot b)\cdot\alpha(c)$.
\end{definition}

Since $\alpha$ twists the associativity, it is referred to as the \emph{twisting map}, and unless otherwise stated, it is understood that $\alpha$ without any further reference will always denote the twisting map of a hom-associative algebra. A \emph{multiplicative} hom-associative algebra is an algebra where the twisting map is multiplicative, i.e. an $R$-algebra homomorphism.

\begin{remark}
A hom-associative algebra over $R$ is in particular a non-unital, non-associative $R$-algebra, and in case $\alpha=\mathrm{id}_M$, a non-unital, associative $R$-algebra. In case $\alpha=0_M$, the hom-associative condition becomes null, and hom-associative algebras can thus be considered as generalizations of both associative and non-associative algebras.
\end{remark}

\begin{definition}[Hom-associative ring]A \emph{hom-associative ring} is a hom-associative algebra over the integers.
\end{definition}

\begin{definition}[Weakly unital hom-associative algebra]\label{def:weak-hom}
Let $A$ be a hom-associative algebra. If for all $a\in A$, $e_l\cdot a=\alpha(a)$ for some $e_l\in A$, we say that $A$ is \emph{weakly left unital} with \emph{weak left unit} $e_l$. In case $a\cdot e_r=\alpha(a)$ for some $e_r\in A$, $A$ is called \emph{weakly right unital} with \emph{weak right unit} $e_r$. If there is an $e\in A$ which is both a weak left and a weak right unit, $e$ is called a \emph{weak unit}, and $A$ \emph{weakly unital}.
\end{definition}

\begin{remark}The notion of a weak unit can thus be seen as a weakening of that of a unit. A weak unit, when it exists, need not be unique.
\end{remark}

\begin{proposition}[\cite{Yau09, FG09}]\label{prop:star-alpha-mult} Let $A$ be a unital, associative algebra with unit $1_A$, $\alpha$ an algebra endomorphism on $A$, and define $*\colon A\times A\to A$ by $a* b:=\alpha(a\cdot b)$ for all $a,b\in A$. Then $(A,*,\alpha)$ is a weakly unital hom-associative algebra with weak unit $1_A$.
\end{proposition}

From now on we will always refer to the construction above when writing $*$.

\begin{definition}[Hom-Lie algebra]\label{def:hom-lie} A \emph{hom-Lie algebra} over an associative, commutative, and unital ring $R$, is a triple $(M, [\cdot,\cdot], \alpha)$ where $M$ is an $R$-module, $\alpha\colon M\to M$ a linear map called the \emph{twisting map}, and $[\cdot,\cdot]\colon M\times M\to M$ a map called the \emph{hom-Lie bracket}, satisfying the following axioms for all $a,b,c\in M$ and $r,s\in R$:
\begin{align*}
[ra+sb,c]=r[a,c]+s[b,c], \quad [a,rb+sc]=r[a,b]+s[a,c],\quad\text{(bilinearity)},\\
[a,a]=0,\quad\text{(alternativity)},\\
\left[\alpha(a),[b,c]\right]+\left[\alpha(c),[a,b]\right]+\left[\alpha(b),[c,a]\right]=0,\quad\text{(hom-Jacobi identity)}.\label{eq:hom-jacobi}
\end{align*}
\end{definition}
As in the case of Lie algebra, we immediately get anti-commutativity from the bilinearity and alternativity by calculating $0=[a+b,a+b]=[a,a]+[a,b]+[b,a]+[b,b]=[a,b]+[b,a]$, so $[a,b]=-[b,a]$ holds for all $a$ and $b$ in a hom-Lie algebra as well. Unless $R$ has characteristic two, anti-commutativity also implies alternativity, since $[a,a]=-[a,a]$ for all $a$.

\begin{remark}If $\alpha=\mathrm{id}_M$ in \autoref{def:hom-lie}, we get the definition of a Lie algebra. Hence the notion of a hom-Lie algebra can be seen as a generalization of that of a Lie algebra.
\end{remark}

\begin{proposition}[\cite{MS08}]\label{prop:commutator-construction} Let $(M,\cdot,\alpha)$ be a hom-associative algebra with commutator $[\cdot,\cdot]$. Then $(M,[\cdot,\cdot],\alpha)$ is a hom-Lie algebra.
\end{proposition}
Note that when $\alpha$ is the identity map in the above proposition, one recovers the classical construction of a Lie algebra from an associative algebra. We refer to the above construction as the \emph{commutator construction}.

\subsection{Non-unital, hom-associative Ore extensions}\label{subsec:hom-ore}
Here, we give some preliminaries from the theory of non-unital, hom-associative Ore extensions, as introduced in~\cite{BRS18}. First, if $R$ is a non-unital, non-associative ring, a map $\beta\colon R\to R$ is called \emph{left $R$-additive} if for all $r,s,t\in R$, we have $r\cdot\beta(s+t)=r\cdot \beta(s)+r\cdot\beta(t)$. If given two left $R$-additive maps $\delta$ and $\sigma$ on a non-unital, non-associative ring $R$, by a \emph{non-unital, non-associative Ore extension} of $R$, written $R[x;\sigma,\delta]$, we mean the set of formal sums $\sum_{i\in\mathbb{N}}r_ix^i$ where finitely many $r_i\in R$ are non-zero, equipped with the following addition:
\begin{equation}
\sum_{i\in\mathbb{N}}r_ix^i+\sum_{i\in\mathbb{N}}s_ix^i=\sum_{i\in\mathbb{N}}(r_i+s_i)x^i,\quad r_i,s_i\in R,
\end{equation}
and the following multiplication, first defined on \emph{monomials} $rx^m$ and $sx^n$ where $m,n\in\mathbb{N}$:
\begin{equation}
rx^m\cdot sx^n=\sum_{i\in\mathbb{N}}\left(r\cdot\pi_i^m(s\right))x^{i+n}\label{eq:ore-mult}
\end{equation}
and then extended to arbitrary \emph{polynomials} $\sum_{i\in\mathbb{N}}r_ix^i$ in $R[x;\sigma,\delta]$ by imposing distributivity. The functions $\pi_i^m\colon R\to R$ are defined as the sum of all $\binom{m}{i}$ compositions of $i$ instances of $\sigma$ and $m-i$ instances of $\delta$, so $\pi_2^3=\sigma\circ\sigma\circ\delta+\sigma\circ\delta\circ\sigma+\delta\circ\sigma\circ\sigma$ and $\pi_0^0=\mathrm{id}_R$. Whenever $i<0$, or $i>m$, we put $\pi_i^m\equiv 0$. That this really gives an extension of the ring $R$, as suggested by the name, can now be seen by the fact that $rx^0\cdot sx^0=\sum_{i\in\mathbb{N}}(r\cdot \pi_i^0(s))x^{i+0}=(r\cdot\pi_0^0(s))x^0=(r\cdot s)x^0$, and similarly $rx^0+sx^0=(r+s)x^0$ for any $r,s\in R$. Hence the isomorphism $r\mapsto rx^0$ embeds $R$ into $R[x;\sigma,\delta]$. We shall only be concerned with the case $\sigma=\mathrm{id}_R$, however, in which case \eqref{eq:ore-mult} simplifies to
\begin{equation}
rx^m\cdot sx^n=\sum_{i\in\mathbb{N}}\binom{m}{i}\left(r\cdot\delta^{m-i}(s)\right)x^{i+n}.
\end{equation}
Starting with a non-unital, non-associative ring $R$ equipped with two left $R$-additive maps $\delta$ and $\sigma$ and some additive map $\alpha\colon R\to R$, we \emph{extend $\alpha$ homogeneously} to $R[x;\sigma, \delta]$ by putting $\alpha(rx^m)=\alpha(r)x^m$ for all $rx^m\in R[x;\sigma,\delta]$, imposing additivity. If $\alpha$ is further assumed to be multiplicative and to commute with $\delta$ and $\sigma$, we can turn a non-unital (unital), associative Ore extension into a non-unital (weakly unital), hom-associative Ore extension by using this extension, as the following proposition demonstrates:

\begin{proposition}[\cite{BRS18}]\label{prop:hom*ore} Let $R[x;\sigma,\delta]$ be a non-unital, associative Ore extension of a non-unital, associative ring $R$. Let $\alpha\colon R\to R$ be a ring endomorphism that commutes with $\delta$ and $\sigma$, and extend $\alpha$ homogeneously to $R[x;\sigma,\delta]$. Then $\left(R[x;\sigma,\delta],*,\alpha\right)$ is a multiplicative, non-unital, hom-associative Ore extension.
\end{proposition}

\begin{remark} Note that if $S:=R[x;\sigma,\delta]$ in \autoref{prop:hom*ore} is unital with unit $1_S$, then $(S,*,\alpha)$ is weakly unital with weak unit $1_S$ by \autoref{prop:star-alpha-mult}.
\end{remark}

\subsection{The first Weyl algebra}\label{subsec:weyl-algebra}
The first Weyl algebra $A_1$ over a field $K$ of characteristic zero, from now on just referred to as the Weyl algebra, is the free, associative and unital algebra on two letters $x$ and $y$, $K\langle x,y\rangle$, modulo the commutation relation $[x,y]:=x\cdot y - y\cdot x = 1_{A_1}$, $1_{A_1}=1_Ky^0x^0$ being the unit element in $A_1$. It may be exhibited as the associative and unital iterated Ore extension $K[y][x;\mathrm{id}_{K[y]},\mathrm{d}/\mathrm{d}y]$ where $\mathrm{d}/\mathrm{d}y$ is the ordinary derivative on $K[y]$ (see e.g. \cite{GW04}, also giving a nice introduction to the theory of (associative) Ore extensions). Furthermore, $A_1$ is a classical example of a non-commutative domain, and as a vector space it has a basis $\{y^ix^j\colon i,j\in\mathbb{N}\}$. The fact that $A_1$ does not contain any zero divisors implies that any nonzero endomorphism $f$ on $A_1$ is unital, since $f(1_{A_1})=f(1_{A_1})\cdot f(1_{A_1})\iff f(1_{A_1})\cdot(1_{A_1}-f(1_{A_1}))=0\implies f(1_{A_1})=1_{A_1}$.

Littlewood~\cite{Lit33} proved that $A_1$ is simple when $K=\mathbb{R}$ and $K=\mathbb{C}$, and Hirsch~\cite{Hir37} then generalized this to when $K$ is an arbitrary field of characteristic zero, as well as for higher order Weyl algebras. Sridharan showed in~\cite{Sri61} (cf. Remark 6.2 and Theorem 6.1) that the cohomology of $A_1$ is zero in all positive degrees (see also Theorem 5 in \cite{GG14}). In particular, the vanishing of the cohomology in the first and second degree imply that all derivations are inner and that $A_1$ is formally rigid in the classical sense of Gerstenhaber~\cite{Ger64}. It should be mentioned that there exists however a nontrivial so-called non-commutative deformation, which is due to Pinczon~\cite{Pin97}. In this deformation, the deformation parameter no longer commutes with the original algebra, making it possible to deform $A_1$ into $U(\mathfrak{osp}(1,2))$, the universal enveloping algebra of the orthosymplectic Lie superalgebra (cf. Proposition 4.5 in~\cite{Pin97}). Dixmier further undertook a thorough investigation of $A_1$ in~\cite{Dix68}, and in the same paper he asked whether all its endomorphisms are actually automorphisms? Although this question still remains unanswered, Dixmier managed to describe its automorphism group $\aut_K(A_1)$, which Makar-Limanov then gave a new proof of~\cite{Mak84}, describing the generators of $\aut_K(A_1)$ as follows:

\begin{theorem}[\cite{Mak84}]\label{thm:Mak84} $\aut_K(A_1)$ is generated by \emph{linear automorphisms}, 
\begin{equation*}
x\mapsto ax+by,\quad y\mapsto cx+dy, \quad \begin{vmatrix}a&c\\b&d\end{vmatrix}=1,
\end{equation*}
for $a,b,c,d\in K$, and \emph{triangular automorphisms},
\begin{equation*}
x\mapsto x,\quad y\mapsto y+p(x),\quad p(x)\in K[x].
\end{equation*}
\end{theorem}

\subsection{The hom-associative Weyl algebras}\label{subsec:hom-weyl} In \cite{BRS18}, a family of \emph{hom-associative Weyl algebras} $\{A_1^k\}_{k\in K}$ were constructed as generalizations of $A_1$ to the hom-associative setting, including $A_1$ in the member corresponding to $k=0$; $A^0_1=A_1$. Concretely, one finds that an endomorphism $\alpha_k$ commutes with $\mathrm{d}/\mathrm{d}y$ (and $\mathrm{id}_{K[y]}$) on $K[y]$ if and only if it is of the form $\alpha_k(y)=y+k$, $\alpha_k(1_{A_1})=1_{A_1}$, for some arbitrary $k\in K$. Hence, in the light of \autoref{prop:hom*ore}, we have for each $k\in K$ a hom-associative Weyl algebra $A_1^k=(A_1,*,\alpha_k)$, where $\alpha_k(x)=x$. The unit element $1_{A_1}$ of $A_1$ is now acting as a weak unit in the whole of $A_1^k$, where for instance $1_{A_1}*y:=\alpha_k(1_{A_1}\cdot y)=\alpha_k(y)=y+k$. Moreover, if we use the subscript $*$ whenever the multiplication is that defined in \autoref{prop:star-alpha-mult}, so that $(\cdot,\cdot,\cdot)_*$ is the associator and $[\cdot,\cdot]_*$ the commutator in $A_1^k$, we have $[x,y]_*=[x,y]=1_{A_1}$. It was further shown in \cite{BRS18} that $A_1^k$ is simple for all $k\in K$.

\section{Morphisms, derivations, commutation and association relations}\label{sec:morph-assoc-comm}
This section contains results of some of the basic properties of $A_1^k$. 

\begin{lemma}\label{lem:preserved-weak-units}Surjective morphisms of hom-associative algebras preserve weak left (right) units.
\end{lemma}

\begin{proof}Let $f\colon A\to B$ be a surjective morphism between two hom-associative algebras with twisting maps $\alpha_A$ and $\alpha_B$, respectively, and $e_A$ a weak left unit of $A$. We show the left case; the right case is analogous. For any element $b\in B$, there is an $a\in A$ such that $b=f(a)$, so $f(e_A)\cdot b=f(e_A)\cdot f(a)=f(e_A\cdot a)=f(\alpha_A(a))=\alpha_B(f(a))=\alpha_B(b)$.
\end{proof}

\begin{proposition}\label{prop:subfield-of-hom-weyl} $K$ embeds as a subfield into $A_1^k$.
\end{proposition}

\begin{proof}$K$ is embedded into the associative Weyl algebra $A_1$ by the isomorphism $f\colon K\to  K':=\{ay^0x^0\colon a\in K\}\subseteq A_1$ defined by $f(a)=ay^0x^0$ for any $a\in K$. One readily verifies that the same map embeds $K$ into $A_1^k$, i.e. it is also an isomorphism of the hom-associative algebra $K$, the twisting map being $\mathrm{id}_K$, and the hom-associative subalgebra $K'\subseteq A_1^k$. 
\end{proof}
Just as in the associative case, the above proposition makes it possible to identify $ay^0x^0$ with $a$ for any $a\in K$, something we will do from now on.
\begin{lemma}\label{lem:unique-weak-unit}$1_{A_1}$ is a unique weak left and a unique weak right unit of $A_1^k$.
\end{lemma}

\begin{proof} First note that $\alpha_k$ is injective on $A_1^k$; it is injective on $A_1$, and since the underlying vector space is the same for the two algebras, also injective on $A_1^k$. Assume $e_l\in A_1^k$ is a weak left unit. Then $e_l*1_{A_1}=\alpha_k(1_{A_1})$. Since $1_{A_1}$ is a weak right unit, $e_l*1_{A_1}=\alpha_k(e_l)$, so $\alpha_k(e_l)=\alpha_k(1_{A_1})$. Hence $e_l=1_{A_1}$, and analogously for the right case. 
\end{proof}

On $A_1$ we may define partial differential operators by $\frac{\partial}{\partial y} (y^mx^n) := my^{m-1}x^n$, $\frac{\partial}{\partial x} (y^mx^n) := ny^{m}x^{n-1}$ for any $m,n\in\mathbb{N}$, $0y^{-1}x^n$ and $0y^mx^{-1}$ defined to be zero, and extending linearly. If $L$ is some linear operator on $A_1^k$ such that for each $p\in A_1^k$, only finitely many elements $L^ip$ for $i\in\mathbb{N}$ are nonzero, then we may define $e^L$ using ordinary formal power series. The next proposition gives an example of that.

\begin{proposition}[Product and twisting map]\label{prop:deformed-product} $\alpha_k=e^{k\frac{\partial}{\partial y}}$, so for all $p,q\in A_1^k$,
\begin{equation}
p*q=e^{k\frac{\partial}{\partial y}}(p\cdot q).\label{eq:deformed-product}
\end{equation}
\end{proposition}

\begin{proof}Put $p = \sum_{i\in\mathbb{N}}\sum_{j\in\mathbb{N}}p_{ij}y^ix^j$ for some $p_{ij}\in K$. Then, by defining the exponential of the partial differential operator as its formal power series and putting $0y^{i}$ to be zero whenever $i<0$,
\begin{align*}
\alpha_k(p)&=\sum_{i\in\mathbb{N}}\sum_{j\in\mathbb{N}}p_{ij}(y+k)^ix^j=\sum_{i\in\mathbb{N}}\sum_{j\in\mathbb{N}}\sum_{l=0}^ip_{ij}\binom{i}{l}k^ly^{i-l}x^j\\
&=\sum_{i\in\mathbb{N}}\sum_{j\in\mathbb{N}}\sum_{l=0}^ip_{ij}\left(\left(k\frac{\partial}{\partial y}\right)^l\Big/l!\right)y^ix^j=\sum_{i\in\mathbb{N}}\sum_{j\in\mathbb{N}}\sum_{l\in\mathbb{N}}p_{ij}\left(\left(k\frac{\partial}{\partial y}\right)^l\Big/l!\right)y^ix^j\\
&=\sum_{l\in\mathbb{N}}\left(\left(k\frac{\partial}{\partial y}\right)^l\Big/l!\right)\sum_{i\in\mathbb{N}}\sum_{j\in\mathbb{N}}p_{ij}y^ix^j=:e^{k\frac{\partial}{\partial y}}p.
\end{align*}
At last, $p*q:=\alpha_k(p\cdot q)$, and hence \eqref{eq:deformed-product} holds for all $p,q\in A_1^k$.
\end{proof}

\begin{remark}\label{re:inverse}The inverse of $e^{k\frac{\partial}{\partial y}}$ is simply $e^{-k\frac{\partial}{\partial y}}$, so by \eqref{eq:deformed-product}, $p\cdot q= e^{-k\frac{\partial}{\partial y}}(p*q)$.
\end{remark}

\begin{corollary}\label{cor:zero-div}There are no zero divisors in $A_1^k$.
\end{corollary}

\begin{proof}Using \autoref{re:inverse}, $A_1^k$ cannot contain any zero divisors since $A_1$ does not.
\end{proof}

\begin{corollary}[Commutation relations]\label{cor:comm-rel}For any polynomial $p(x,y) \in A_1^k$,
\begin{align}
\left[x,p(x,y)\right]_*&=e^{k\frac{\partial}{\partial y}}[x,p(x,y)]=\frac{\partial}{\partial y} p(x,y+k),\label{eq:x-commute}\\
\left[p(x,y),y\right]_*&=e^{k\frac{\partial}{\partial y}}[p(x,y),y]=\frac{\partial}{\partial x}p(x,y+k). \label{eq:y-commute}
\end{align}
\end{corollary}

\begin{proof}In $A_1$, we have $[x,y^mx^n]=x\cdot y^mx^n-y^mx^n\cdot x = \sum_{i\in\mathbb{N}}\binom{1}{i}\frac{\partial^{1-i}y^m}{\partial y^{1-i}}x^{n+i}-y^mx^{n+1}=my^{m-1}x^n$ for any $m,n\in\mathbb{N}$, defining $0y^{-1}$ to be zero. By linearity in the second argument, it follows that $[x,p(x,y)]=\frac{\partial}{\partial y}p(x,y)$. By using \autoref{prop:deformed-product},
\begin{align*}
[x,p(x,y)]_*&=x*p(x,y)-p(x,y)*x=\alpha_k(x\cdot p(x,y))-\alpha_k(p(x,y)\cdot x)\\
&=e^{k\frac{\partial}{\partial y}}(x\cdot p(x,y)-p(x,y)\cdot x)=e^{k\frac{\partial}{\partial y}}[x,p(x,y)]=e^{k\frac{\partial}{\partial y}}\frac{\partial}{\partial y}p(x,y)\\
&=\frac{\partial}{\partial y}e^{k\frac{\partial}{\partial y}} p(x,y)=\frac{\partial}{\partial y}p(y+k,x).
\end{align*}
In $A_1$, we also have $\left[y^mx^n,y\right]=y^mx^n\cdot y - y\cdot y^mx^n=y^m\sum_{i\in\mathbb{N}}\binom{n}{i}\frac{\partial^{n-i}y}{\partial y^{n-i}}x^i - y^{m+1}x^n=ny^mx^{n-1}$ for any $m,n\in\mathbb{N}$, defining $0x^{-1}$ to be zero. By linearity in the first argument, it follows that $[p(x,y),y]=\frac{\partial}{\partial x}p(x,y)$. Hence,
\begin{equation*}
[p(x,y),y]_*=e^{k\frac{\partial}{\partial y}}[p(x,y),y]=e^{k\frac{\partial}{\partial y}}\frac{\partial}{\partial x}p(x,y)=\frac{\partial}{\partial x}e^{k\frac{\partial}{\partial y}}p(x,y)=\frac{\partial}{\partial x}p(y+k,x).\qedhere
\end{equation*}
\end{proof}

\begin{proposition}[The commuter]\label{prop:commuter-of-hom-weyl} $C(A_1^k)=K$.
\end{proposition}

\begin{proof} Let $a\in K$ and $q\in A_1^k$ be arbitrary. Then $[a,q]_*=\alpha_k\left([a,q]\right)=\alpha_k(0)=0$, so $K\subseteq C(A_1^k)$. For any $p\in C(A_1^k)$, $[x,p]_*\stackrel{\eqref{eq:x-commute}}{=}e^{k\frac{\partial}{\partial y}}[x,p]\stackrel{!}{=}0$, which implies $[x,p]=0$. From \autoref{cor:comm-rel}, $[x,p]=\frac{\partial}{\partial y}p$, so $p\in K[x]$. Continuing, $[p,y]_*\stackrel{\eqref{eq:y-commute}}{=}e^{k\frac{\partial}{\partial y}}[p,y]\stackrel{!}{=}0$, which implies $[p,y]=0$. Again,  from \autoref{cor:comm-rel}, $[p,y]=\frac{\mathrm{d}}{\mathrm{d} x}p$, so $p\in K$.
\end{proof}

\begin{corollary}[The center]\label{cor:center} 
\begin{equation*}
Z(A_1^k)=\begin{cases}K&\text{if } k=0,\\ \{0\}&\text{otherwise.}\end{cases}
\end{equation*}
\end{corollary}

\begin{proof} Recall from \autoref{subsec:general-algebra} that $Z(A_1^k)=C(A_1^k)\cap N(A_1^k)$. When $k=0$, $N(A_1^k)=A_1^k$, and hence $Z(A_1^k)=C(A_1^k)=K$. Assume instead that $k\neq0$, and let $c\in K$ be arbitrary. Then a straightforward calculation yields $(c,y,y)_*=-2ck^2-cky\stackrel{!}{=}0\iff c=0$. On the other hand, $0\in N(A_1^k)$, so $Z(A_1^k)=\{0\}$.
\end{proof}

\begin{proposition}[Power associativity]\label{prop:power-assoc} $A_1^k$ is power associative if and only if $k=0$.
\end{proposition}

\begin{proof} If $k=0$, then $A_1^k$ is associative and hence also power associative. On the other hand, one readily verifies that $(yx,yx,yx)_*=kx+2k^2x^2$, so if $A_1^k$ is power associative, then $k=0$.
\end{proof}

\begin{remark}
Note that due to the proposition above, $A_1^k$ is not left alternative, right alternative, or flexible, let alone associative. 
\end{remark}

In \autoref{subsec:general-algebra} we said that maps of the form $[a,\cdot]\colon A\to A$ for any $a$ in an associative algebra $A$ are derivations called inner derivations, and that such a map need not be a derivation if $A$ is not associative. For a concrete example of this latter fact, one can consider the map $[y^2,\cdot]_*$ in $A_1^k$, which is a derivation if and only if $k=0$. The reason for this failure when $k=0$ is due to the next lemma.

\begin{lemma}\label{lem:derivation-equiv}$\delta$ is a derivation on $A_1^k$ if and only if $\delta$ is a derivation on $A_1$ that commutes with $e^{k\frac{\partial}{\partial y}}$.
\end{lemma}

\begin{proof}First, note that $\delta$ is a linear map on $A_1^k$ if and only if it is a linear map on $A_1$,  as the underlying vector space of $A_1^k$ and $A_1$ is the same. Now, let $\delta$ be a derivation on $A_1^k$. We claim that $\delta(1_{A_1})=0$. First, $\delta(1_{A_1}*1_{A_1})=\delta(1_{A_1})*1_{A_1}+1_{A_1}*\delta(1_{A_1})=2\alpha_k(\delta(1_{A_1}))=2e^{k\frac{\partial}{\partial y}}\delta(1_{A_1})$, using that $\alpha_k=e^{k\frac{\partial}{\partial y}}$ from \autoref{prop:deformed-product}. On the other hand, $\delta(1_{A_1}*1_{A_1})=\delta(\alpha_k(1_{A_1}))=\delta(1_{A_1})$.  The equality of the two expressions is then equivalent to the eigenvector problem $e^{k\frac{\partial}{\partial y}}p=\frac{1}{2}p$, where $p=\delta(1_{A_1})$. It turns out it has no solution, which may be seen from solving the equivalent PDE $p+2\left(k\frac{\partial}{\partial y}+\frac{k^2}{2!}\frac{\partial^2}{\partial y^2}+\dots+\frac{k^m}{m!}\frac{\partial^m}{\partial y^m}\right)p=0$. To see this, let us put $p=\sum_{i=0}^m\sum_{j=0}^np_{ij}y^ix^j$ for some $p_{ij}\in K$ and $m,n\in\mathbb{N}$. Then, by comparing coefficients, starting with $p_{mj}$ for some arbitrary $j$ and working our way down to $p_{0j}$, we have that $p_{ij}=0$ for all $i,j\in\mathbb{N}$. Therefore, $\delta(1_{A_1})=0$ as claimed. For arbitrary $q\in A_1^k$, $\delta\left(e^{k\frac{\partial}{\partial y}}q\right)=\delta(\alpha_k(q))=\delta(q*1_{A_1})=\delta(q)*1_{A_1}+q*\delta(1_{A_1})=\delta(q)*1_{A_1}=\alpha_k(\delta(q))=e^{k\frac{\partial}{\partial y}}\delta(q)$, so $\delta$ commutes with $e^{k\frac{\partial}{\partial y}}$. Now, $\alpha_k(\delta(r\cdot s))=e^{k\frac{\partial}{\partial y}}\delta(r\cdot s)=\delta\left(e^{k\frac{\partial}{\partial y}}(r\cdot s)\right)=\delta(\alpha_k(r\cdot s))=\delta(r*s)=\delta(r)*s+r*\delta(s)=\alpha_k(\delta(r)\cdot s)+\alpha_k(r\cdot\delta(s))=\alpha_k(\delta(r)\cdot s + r\cdot\delta(s))$ where $r,s\in A_1$ are arbitrary. By the injectivity of $\alpha_k$, $\delta(r\cdot s)=\delta(r)\cdot s + r\cdot\delta(s)$. Assume now instead that $\delta$ is a derivation on $A_1$ that commutes with $e^{k\frac{\partial}{\partial y}}$, and that $r,s\in A_1^k$. Then, $\delta(r*s)=\delta(\alpha_k(r\cdot s))=\delta\left(e^{k\frac{\partial}{\partial y}}(r\cdot s)\right)=e^{k\frac{\partial}{\partial y}}\delta(r\cdot s)=\alpha_k(\delta(r\cdot s))=\alpha_k(\delta(r)\cdot s + r\cdot\delta(s))=\alpha_k(\delta(r)\cdot s)+\alpha_k(r\cdot\delta(s))=\delta(r)*s+r*\delta(s)$.\qedhere
\end{proof}

\begin{corollary}[Derivations]\label{cor:derivations} $\delta$ is a derivation on $A_1^k$ for $k$ nonzero if and only if $\delta=[cy+p(x),\cdot]=e^{-k\frac{\partial}{\partial y}}[cy+p(x),\cdot]_*$ for some $c\in K$ and $p(x)\in K[x]$.
\end{corollary}

\begin{proof}Recall from \autoref{subsec:weyl-algebra} that all derivations on $A_1$ are inner, i.e. of the form $[q,\cdot]$ for some $q\in A_1$. From \autoref{lem:derivation-equiv}, there is a one-to-one correspondence between the derivations on $A_1^k$ and the derivations on $A_1$ that commute with $e^{k\frac{\partial}{\partial y}}$. Hence, we are looking for $q\in A_1$ such that  $e^{k\frac{\partial}{\partial y}}[q,x]=\left[q,e^{k\frac{\partial}{\partial y}}x\right]=[q,x]$ and $e^{k\frac{\partial}{\partial y}}[q,y]=\left[q,e^{k\frac{\partial}{\partial y}}y\right]=[q,y+k]=[q,y]$. We thus have two eigenvector problems of the form $e^{k\frac{\partial}{\partial y}}s=s$ with $s\in\{[q,x], [q,y]\}$. This is equivalent to the PDE $\left(k\frac{\partial}{\partial y}+\frac{k^2}{2!}\frac{\partial^2}{\partial y^2}+\dots+\frac{k^m}{m!}\frac{\partial^m}{\partial y^m}\right)s=0$, and by putting $s=\sum_{i\in\mathbb{N}}\sum_{j\in\mathbb{N}}^ns_{ij}y^ix^j$ for some $s_{ij}\in K$, we see by comparing coefficients that $s=\sum_{j\in\mathbb{N}}s_{0j}x^j$. Now, using that $s\in K[x]$, $[q,x]=-\frac{\partial}{\partial y}q$ and $[q,y]=\frac{\partial}{\partial x}q$ from \autoref{cor:comm-rel}, we get $\frac{\partial}{\partial y}q\in K[x]$ and $\frac{\partial}{\partial x}q\in K[x]$. If we put $q=\sum_{i\in\mathbb{N}}\sum_{j\in\mathbb{N}}q_{ij}y^ix^j$, then the former implies that $q=\sum_{j\in\mathbb{N}}(q_{0j}+q_{1j}y)x^j$, which upon plugging into the second yields $q=q_{10}y+\sum_{j\in\mathbb{N}}q_{0j}x^j$. We also claim that a $q$ of this form is sufficient for fulfilling $e^{k\frac{\partial}{\partial y}}[q,u]=\left[q,e^{k\frac{\partial}{\partial y}}u\right]$ for any $u\in A_1$. First, $e^{k\frac{\partial}{\partial y}} q=kq_{10}+q$. Recalling that $e^{k\frac{\partial}{\partial y}}$ is an endomorphism on $A_1$, $e^{k\frac{\partial}{\partial y}}[q,u]=\left[e^{k\frac{\partial}{\partial y}}q, e^{k\frac{\partial}{\partial y}}u\right]=\left[kq_{10}+q, e^{k\frac{\partial}{\partial y}}u\right]=\left[q, e^{k\frac{\partial}{\partial y}}u\right]$. If $q_{10}:=c$ and $p(x):=\sum_{j\in\mathbb{N}}q_{0j}x^j$, then $q=cy+p(x)$. By \autoref{re:inverse}, $[cy+p(x),\cdot]=e^{-k\frac{\partial}{\partial y}}[cy+p(x),\cdot]_*$.
\end{proof}


\begin{lemma}\label{lem:hom-weyl-homo-iff} $f\colon A_1^k\to A_1^l$ is a homomorphism if and only if $f$ is an endomorphism on $A_1$ such that $e^{l\frac{\partial}{\partial y}}f(x)=f(x)$ and $e^{l\frac{\partial}{\partial y}}f(y)=f(y)+k$.
\end{lemma}

\begin{proof}Let $f\colon A_1^k\to A_1^l$ be a homomorphism, i.e. a $K$-linear map such that $f\circ\alpha_k = \alpha_l\circ f$ and $f(a*_kb)=f(a)*_lf(b)$ for all $a,b\in A_1^k$. Since we may view the underlying vector space of $A_1^k, A_1^l$, and $A_1$ as the same, we only need to show that $e^{l\frac{\partial}{\partial y}}f(x)=f(x)$, $e^{l\frac{\partial}{\partial y}}f(y)=f(y)+k$, and $f(a\cdot b)=f(a)\cdot f(b)$. The former follows from $f\circ\alpha_k = \alpha_l\circ f$ with $\alpha_k=e^{k\frac{\partial}{\partial y}}$ from \autoref{prop:deformed-product}, together with the fact that $f(1_{A_1})=1_{A_1}$ as mentioned in \autoref{subsec:weyl-algebra}. The latter follows from the fact that $f(a*_kb)=f(\alpha_k(a\cdot b))=\alpha_l(f(a\cdot b))$, whereas $f(a)*_lf(b)=\alpha_l\left(f(a)\cdot f(b)\right)$, and since $\alpha_l$ is injective, $f(a\cdot b)=f(a)\cdot f(b)$. Assume instead that $f$ is an endomorphism on $A_1$ such that $e^{l\frac{\partial}{\partial y}}f(x)=f(x)$ and $e^{l\frac{\partial}{\partial y}}f(y)=f(y)+k$. Then, with $\alpha_l=e^{l\frac{\partial}{\partial y}}$, for any $y^mx^n\in A_1$, 
\begin{align*}
\alpha_l(f(y^mx^n))&=\alpha_l(f^m(y))\alpha_l(f^n(x))=(\alpha_l(f(y)))^m(\alpha_l(f(x)))^n\\
&=(f(\alpha_k(y)))^m(f(\alpha_k(x))))^n=f(\alpha_k^m(y)\alpha_k^n(x))=f(\alpha_k(y^mx^n)),
\end{align*}
so $\alpha_l\circ f=f\circ\alpha_k$. Moreover, for all $a,b\in A_1^k$, we have $f(a*_kb)=f(\alpha_k(a\cdot b))=\alpha_l(f(a\cdot b))=\alpha_l(f(a)\cdot f(b))=f(a)*_lf(b)$.
\end{proof}

\begin{proposition}[Morphisms]\label{prop:morphisms} Any homomorphism $f\colon A_1^k\to A_1^l$ for $k,l\neq0$ is an isomorphism of the form $f(x)=\frac{l}{k}x+c$, $f(y)= \frac{k}{l}y+p(x)$ for some $c\in K$ and $p(x)\in K[x]$. 
\end{proposition}

\begin{proof}Let us try to find a homomorphism $f\colon A_1^k \to A_1^l$ when $k$ and $l$ are nonzero. By \autoref{lem:hom-weyl-homo-iff}, this is equivalent to finding an endomorphism $f$ on $A_1$ such that $e^{l\frac{\partial}{\partial y}}f(x)=f(x)$ and $e^{l\frac{\partial}{\partial y}}f(y)=f(y)+k$. The former of the two conditions was considered in the proof of \autoref{cor:derivations}, and it turned out to be equivalent to $f(x)\in K[x]$. The latter is equivalent to the PDE $\left(l\frac{\partial}{\partial y}+\frac{l^2}{2!}\frac{\partial^2}{\partial y^2}+\dots+\frac{l^m}{m!}\frac{\partial^m}{\partial y^m}\right)f(y)=k$. If we put $f(y)=\sum_{i=0}^m\sum_{j=0}^na_{ij}y^ix^j$ for some $a_{ij}\in K$ and $m,n\in\mathbb{N}$, then, by comparing coefficients, $f(y)=\frac{k}{l}y+p(x)$ where $p(x):=\sum_{j=0}^na_{0j}x^j$. Now, note that $f$ is an endomorphism on $A_1$ only if $[f(x),f(y)]=f\left([x,y]\right)=f(1_{A_1})=1_{A_1}$. Calculating the left-hand side, $\left[f(x),\frac{k}{l}y+p(x)\right]=\frac{k}{l}[f(x),y]\stackrel{\eqref{eq:y-commute}}{=}\frac{k}{l}\frac{\mathrm{d}}{\mathrm{d}x}f(x)$, which is equal to $1_{A_1}$ if and only if $f(x)=\frac{l}{k}x+c$ for some $c\in K$. Let us introduce the following functions:
\begin{align*}
g_1(x):=&\frac{l}{k}x+y,&g_2(x):=&x, & g_3(x):=&x-\frac{k}{l}y,&g_4(x):=&x, \\
g_1(y):=&\frac{k}{l}y,&g_2(y):=&y+c,& g_3(y):=&y,& g_4(y):=&y-c+\frac{l}{k}p(x).
\end{align*}
According to \autoref{thm:Mak84}, these are all automorphisms on $A_1$, and moreover, $f=g_4\circ g_3\circ g_2\circ g_1$ since $g_4\circ g_3\circ g_2\circ g_1(x)=\frac{l}{k}x+c=f(x)$ and $g_4\circ g_3\circ g_2\circ g_1(y)=\frac{k}{l}y+p(x)=f(y)$. Hence, $f$ is an automorphism on $A_1$ such that $e^{l\frac{\partial}{\partial y}}f(x)=f(x)$ and $e^{l\frac{\partial}{\partial y}}f(y)=f(y)+k$, and therefore an isomorphism from $A_1^k$ to $A_1^l$.
\end{proof}

\begin{corollary}[Hom-Dixmier]\label{cor:hom-dixmier}Any endomorphism $f$ on $A_1^k$ for $k\neq0$ is an automorphism of the form $f(x)=x+c$ and $f(y)=y+p(x)$ for some $c\in K$ and $p(x)\in K[x]$.
\end{corollary}

\begin{proof}This follows from \autoref{prop:morphisms} with $k=l$.
\end{proof}

\section{One-parameter formal deformations}\label{sec:deform}
\emph{One-parameter formal hom-associative deformations} and \emph{one-parameter formal hom-Lie deformations} were first introduced by Makhlouf and Silvestrov in~\cite{MS10} together with an attempt at describing a compatible cohomology theory in lower degrees. In the multiplicative case, this was later expanded on by Ammar, Ejbehi and Makhlouf in~\cite{AEM11}, and then by Hurle and Makhlouf~\cite{HM18}. Only in this latter paper, treating the multiplicative, hom-associative case, did the cohomology theory include the twisting map $\alpha$ in a natural way. This is indeed essential, as the idea behind these kinds of deformations is to deform not only the multiplication map, or the Lie bracket, but also the twisting map $\alpha$, resulting also in a deformation of the twisted associativity condition and the twisted Jacobi identity, respectively. In the special case when the deformations start from $\alpha$ being the identity map and the multiplication being associative, or the bracket being the Lie bracket, one gets a deformation of an associative algebra into a hom-associative algebra, and in the latter case a deformation of a Lie algebra into a hom-Lie algebra. Perhaps the main motivation for studying these kinds of deformations is that they provide a framework in which some algebras can now be deformed, which otherwise could not when considered as objects of the category of associative algebras, or that of Lie algebras. The first Weyl algebra constitutes such an example; in the classical sense, it is rigid (see e.g. \cite{Sri61,GG14} for a proof of this fact). In this section, we show that the hom-associative Weyl algebras are one-parameter formal hom-associative deformations of the first Weyl algebra, and that they induce formal deformations of the corresponding Lie algebras into hom-Lie algebras, when using the commutator as bracket. Here, we use a slightly more general approach than that given in~\cite{MS10}, replacing vector spaces by modules; this follows our convention in the preliminaries and previous work (cf.~\cite{Bac18,BR18,BRS18}), with the advantage of e.g. being able to treat rings as algebras over the integers. First, if $R$ is an associative, commutative, and unital ring, and $M$ an $R$-module, we denote by $R\llbracket t\rrbracket$ the formal power series ring in the indeterminate $t$, and by $M\llbracket t\rrbracket$ the $R\llbracket t\rrbracket$-module of formal power series in the same indeterminate, but with coefficients in $M$. By \autoref{def:hom-assoc-algebra}, this allows us to define a hom-associative algebra $(M\llbracket t\rrbracket, \cdot_t,\alpha_t)$ over $R\llbracket t\rrbracket$.

\begin{definition}[One-parameter formal hom-associative deformation] A \emph{one-pa\-rameter formal hom-associative deformation} of a hom-associative algebra, $(M,\cdot_0,\alpha_0)$ over $R$, is a hom-associative algebra $(M\llbracket t\rrbracket, \cdot_t,\alpha_t)$ over $R\llbracket t\rrbracket$, where
\begin{equation*}
\cdot_t=\sum_{i\in\mathbb{N}} \cdot_i t^i,\quad \alpha_t=\sum_{i\in\mathbb{N}} \alpha_it^i,
\end{equation*}
and for each $i\in\mathbb{N}$, $\cdot_i\colon M\times M\to M$ is a binary operation linear over $R$ in both arguments, and $\alpha_i\colon M\to M$ an $R$-linear map. We further extend $\cdot_i$ homogeneously to a binary operation linear over $R\llbracket t\rrbracket$ in both arguments, $\cdot_i\colon M\llbracket t\rrbracket\times M\llbracket t\rrbracket\to M\llbracket t\rrbracket$, and $\alpha_i$ to an $R\llbracket t\rrbracket$-linear map $\alpha_i\colon M\llbracket t\rrbracket\to M\llbracket t\rrbracket$.
\end{definition}
Here, and onwards, a homogeneous extension is defined analogously to that of an Ore extension in \autoref{subsec:hom-ore}, so that for any $r,s\in R$, $a,b\in M$, and $i,j,l\in\mathbb{N}$, we have $\alpha_i(rat^j+sbt^l)=r\alpha_i(a)t^j + s\alpha_i(b)t^l$, and similarly for the product $\cdot_i$.

\begin{proposition}\label{prop:weyl-deform} $A_1^k$ is a one-parameter formal hom-associative deformation of $A_1$.
\end{proposition}

\begin{proof}We put $t:=k$, and regard $t$ as an indeterminate of the formal power series $K\llbracket t\rrbracket$ and $A_1\llbracket t\rrbracket$; this gives a deformation $(A_1\llbracket t\rrbracket,\cdot_t,\alpha_t)$ of $(A_1,\cdot_0,\mathrm{id}_{A_1})$, where the latter is $A_1$ in the language of hom-associative algebras, $\cdot_0$ denoting the multiplication in $A_1$. Explicitly, with $\alpha_t=e^{t\frac{\partial}{\partial y}}$ from \autoref{prop:deformed-product}, it is clear that $\alpha_t$ is a formal power series in $t$ by definition, and moreover, $\alpha_0=\mathrm{id}_{A_1}$. Next, we extend $\alpha_t$ linearly over $K\llbracket t\rrbracket$ and homogeneously to all of $A_1\llbracket t\rrbracket$. To define the multiplication $\cdot_t$ in $A_1\llbracket t \rrbracket$, we first extend $\cdot_0\colon A_1\times A_1\to A_1$ homogeneously to a binary operation $\cdot_0\colon A_1\llbracket t\rrbracket\times A_1\llbracket t\rrbracket\to A_1\llbracket t\rrbracket$ linear over $K\llbracket t\rrbracket$ in both arguments, and then simply compose $\alpha_t$ with $\cdot_0$, so that $\cdot_t:=\alpha_t\circ\cdot_0=e^{t\frac{\partial}{\partial y}}\circ\cdot_0$. This is again a formal power series in $t$ by definition, and hom-associativity now follows from \autoref{prop:star-alpha-mult}. 
\end{proof}

From now on, we refer to one-parameter formal hom-associative deformations as just \emph{deformations}.

\begin{definition}[One-parameter formal hom-Lie deformation]A \emph{one-parameter formal hom-Lie deformation} of a hom-Lie algebra $(M,[\cdot,\cdot]_0,\alpha_0)$ over $R$ is a hom-Lie algebra $(M\llbracket t\rrbracket, [\cdot,\cdot]_t,\alpha_t)$ over $R\llbracket t\rrbracket$, where
\begin{equation*}
[\cdot,\cdot]_t=\sum_{i\in\mathbb{N}} [\cdot,\cdot]_i t^i,\quad \alpha_t=\sum_{i\in\mathbb{N}} \alpha_it^i,
\end{equation*}
and for each $i\in\mathbb{N}$, $[\cdot,\cdot]_i\colon M\times M\to M$ is a binary operation linear over $R$ in both arguments, and $\alpha_i\colon M\to M$ an $R$-linear map. We further extend $[\cdot,\cdot]_i$ homogeneously to a binary operation linear over $R\llbracket t\rrbracket$ in both arguments, $[\cdot,\cdot]_i\colon M\llbracket t\rrbracket\times M\llbracket t\rrbracket\to M\llbracket t\rrbracket$, and $\alpha_i$ to an $R\llbracket t\rrbracket$-linear map $\alpha_i\colon M\llbracket t\rrbracket\to M\llbracket t\rrbracket$.
\end{definition}

\begin{remark}Alternativity of $[\cdot,\cdot]_t$ is equivalent to alternativity of $[\cdot,\cdot]_i$ for all $i\in\mathbb{N}$.
\end{remark}

\begin{proposition}\label{prop:weyl-lie-deform}The deformation of $A_1$ into $A_1^k$ induces a one-parameter formal hom-Lie deformation of the Lie algebra of $A_1$ into the hom-Lie algebra of $A_1^k$, when using the commutator as bracket.
\end{proposition}

\begin{proof}Using the deformation of $A_1$ into $A_1^k$ in \autoref{prop:weyl-deform}, we put $t:=k$; this gives a deformation $(A_1\llbracket t\rrbracket, [\cdot,\cdot]_t,\alpha_t)$ of $(A_1,[\cdot,\cdot]_0,\mathrm{id}_{A_1})$, where the latter is the Lie algebra of $A_1$ obtained from the commutator construction with $[\cdot,\cdot]_0$ as the commutator. To see this, we first note that by construction, $\alpha_t$ is the same map as defined in the proof of \autoref{prop:weyl-deform}. Hence, we only need to show that $[\cdot,\cdot]_t$ is a deformation of the commutator $[\cdot,\cdot]_0$, and that the hom-Jacobi identity is satisfied. We first extend  $[\cdot,\cdot]_0\colon A_1\times A_1\to A_1$ homogeneously to a binary operation $[\cdot,\cdot]_0\colon A_1\llbracket t\rrbracket\times A_1\llbracket t\rrbracket\to A_1\llbracket t \rrbracket$ linear over $K\llbracket t\rrbracket$ in both arguments. Next, we define $[\cdot,\cdot]_t\colon A_1\llbracket t\rrbracket\times A_1\llbracket t\rrbracket\to A_1\llbracket t \rrbracket$ as  $\alpha_t\circ[\cdot,\cdot]_0=e^{t\frac{\partial}{\partial y}}[\cdot,\cdot]_0$. The hom-Jacobi identity is satisfied by \autoref{prop:commutator-construction} and the construction of $A_1^k$ given in \autoref{subsec:hom-weyl}.
\end{proof}

\section*{Acknowledgements} We wish to thank Joakim Arnlind for a suggestion on classification, and Sergei Silvestrov for some initial discussions. We would also like to thank the referee for helpful comments on how to improve the manuscript.


\newpage
\begin{thebibliography}{99}
\bibitem{AEM11}
F.~Ammar, Z.~Ejbehi, and A.~Makhlouf,
\emph{Cohomology and Deformations of Hom-Algebras},
J. Lie Theory {\bf 21}(4) (2011), pp. 813--836.

\bibitem{Bac18}
P.~B{\"a}ck,
\emph{Notes on formal deformations of quantum planes and universal enveloping algebras},
J. Phys.: Conf. Ser. {\bf 1194}(1) (2019).

\bibitem{BR18}
P.~B{\"a}ck and J.~Richter,
\emph{Hilbert's basis theorem for non-associative and hom-associative Ore extensions},
\texttt{arXiv:1804.11304}.

\bibitem{BRS18}
P.~B{\"ack}, J.~Richter, and S.~Silvestrov,	
\emph{Hom-associative Ore extensions and weak unitalizations},
Int. Electron. J. Algebra {\bf 24} (2018), pp. 174--194.

\bibitem{Dix68}
J.~Dixmier, 
\emph{Sur les algèbres de Weyl},
Bull. Soc. Math. France {\bf 96} (1968), pp. 209--242.

\bibitem{FG09}
Y.~Fregier and A.~Gohr,
\emph{On unitality conditions for Hom-associative algebras},
\texttt{arXiv:0904.4874}.

\bibitem{Ger64}
M.~Gerstenhaber,
\emph{On the Deformation of Rings and Algebras},
Ann. Math. {\bf 79}(1) (1964), pp. 59--103.

\bibitem{GG14}
M.~Gerstenhaber and A.~Giaquinto,
\emph{On the cohomology of the Weyl algebra, the quantum plane, and the $q$-Weyl algebra},
J. Pure Appl. Algebra {\bf 218}(5) (2014), pp. 879-887.

\bibitem{GW04}
K.~R.~Goodearl and R.~B.~Warfield,
\emph{An Introduction to Noncommutative Noetherian Rings},
2. ed., Cambridge University Press, Cambridge U.K., 2004.

\bibitem{HLS06} 
J.~T.~Hartwig, D.~Larsson, and S.~D.~Silvestrov,
\emph{Deformations of Lie algebras using $\sigma$-derivations},
J. Algebra {\bf 295}(2) (2006), pp. 314--361.

\bibitem{Hir37}
K.~A.~Hirsch,
\emph{A Note on Non-Commutative Polynomials},
J. London Math. Soc. {\bf 12}(4) (1937), pp. 264--266.

\bibitem{HM18}
B.~Hurle and A.~Makhlouf,
\emph{$\alpha$-type Hochschild cohomology of Hom-associative algebras and bialgebras}, J. Korean Math. Soc. {\bf 56}(6) (2019), pp. 1655--1687.

\bibitem{KBK07}
A.~Kanel-Belov and M.~Kontsevich,
\emph{The Jacobian conjecture is stably equivalent to the Dixmier conjecture},
Mosc. Math. J. {\bf 7}(2) (2007), pp. 209--218.

\bibitem{Lit33}
D.~E.~Littlewood,
\emph{On the Classification of Algebras},
Proc. London Math. Soc. {\bf 35}(1) (1933), pp. 200--240.

\bibitem{Mak84}
L.~Makar-Limanov, 
\emph{On automorphisms of Weyl algebra},
Bull. Soc. Math. France {\bf 112} (1984), pp. 359--363.

\bibitem{MS08}
A.~Makhlouf and S.~D.~Silvestrov,
\emph{Hom-algebra structures}, J. Gen. Lie Theory Appl. {\bf 2}(2) (2008), pp. 51--64.

\bibitem{MS10}
 A.~Makhlouf and S.~Silvestrov,
\emph{Notes on 1-parameter formal deformations of Hom-associative and Hom-Lie algebras}, Forum Math. {\bf 22}(4) (2010), pp. 715--739.

\bibitem{NOR18}
P.~Nystedt, J.~ \"Oinert, and J.~Richter,
\emph{Non-associative Ore extensions},
Isr. J. Math. {\bf 224}(1) (2018), pp. 263--292.

\bibitem{Ore33}
O.~Ore,
\emph{Theory of Non-Commutative Polynomials}, Ann. Math. {\bf 34}(3) (1933), pp. 480--508.

\bibitem{Pin97}
G.~Pinczon,
\emph{Noncommutative Deformation Theory},
Lett. Math. Phys. {\bf 41}(2) (1997), pp. 101--117.

\bibitem{Sri61}
R.~Sridharan,
\emph{Filtered Algebras and Representations of Lie Algebras},
Trans. Amer. Math. Soc. {\bf 100}(3) (1961), pp. 530--550.

\bibitem{Tsuchimoto05}
Y.~Tsuchimoto,
\emph{Endomorphisms of Weyl algebra and $p$-curvatures},
Osaka J. Math. {\bf 42}(2) (2005), pp. 435--452.

\bibitem{Yau09} 
D.~Yau, 
\emph{Hom-algebras and Homology},
J. Lie Theory {\bf 19}(2) (2009), pp. 409--421.

\end{thebibliography}
\end{document}